\documentclass[final]{siamonline1116}

\usepackage{slashbox}
\usepackage[utf8]{inputenc}
\usepackage{bm}
\usepackage{graphicx}
\usepackage{amsmath,mathrsfs}
\usepackage[parfill]{parskip}
\usepackage{xcolor}
\usepackage[english]{babel}
\usepackage{amsfonts}
\usepackage{amssymb}
\usepackage{amsbsy}
\usepackage{hyperref}
\usepackage{enumitem}
\usepackage{epstopdf}
\usepackage{subcaption}
\usepackage{cleveref}

\graphicspath{{./figs/}}
\ifpdf
  \DeclareGraphicsExtensions{.eps,.pdf,.png,.jpg}
\else
  \DeclareGraphicsExtensions{.eps}
\fi

\numberwithin{theorem}{section}

\newsiamremark{remark}{Remark}
\newsiamremark{Example}{Example}

\newcommand{\TheTitle}{Minimax state estimates for abstract Neumann problems}
\newcommand{\ShortTitle}{Estimates for Neumann problems}
\newcommand{\TheAuthors}{Olexander Nakonechnii, Sergiy Zhuk}

\ifpdf
\hypersetup{
  pdftitle={\TheTitle},
  pdfauthor={\TheAuthors}
}
\fi

\headers{\ShortTitle}{\TheAuthors}

\title{{\TheTitle}\thanks{Published in MTA, 2018
}}
\author{Sergiy Zhuk\thanks{IBM Research, Dublin, Ireland, \texttt{\{sergiy.zhuk\}@ie.ibm.com}} \and Olexander Nakonechnii\thanks{T.Shevchenko University of Kyiv, Kyiv, Ukraine, \texttt{nakonechniy@unicyb.kiev.ua}}
}
\usepackage{amsopn}
\begin{document}
\maketitle
\begin{abstract}
The paper presents analytic expressions of minimax (worst-case) estimates for solutions of linear abstract Neumann problems in Hilbert space with uncertain (not necessarily bounded!) inputs and boundary conditions given incomplete observations with stochastic noise. The latter is assumed to have uncertain but bounded correlation operator. It is demonstrated that the minimax estimate is asymptotically exact under mild assumptions on the observation operator and bounding sets. A relationship between the proposed estimates and a robust pseudo-inversion of compact operators is revealed. This relationship is illustrated by an academic numerical example: homogeneous Neumann problem for Poisson equation in two spatial dimensions.
\end{abstract}



\section{Introduction}
\label{sec:introduction}


In this paper we propose a generalisation of the minimax\footnote{Minimax estimation is also reffered to as set-membership or worst-case estimation}) state estimation framework for the case of abstract Neumann problems. More specifically, we study existence, uniqueness and provide various representations of the minimax estimates of the states (solutions) $\varphi$ of a linear operator equation $N\varphi = B_1 f_1$, $\delta\varphi=B_0f_0$ with uncertain ``boundary condition'' $f_0$ and uncertain input $f_1$ from a given bounding set $G$. Operators $N$ and $\delta$ are associated with the abstract Neumann problem: they could, for instance, represent an elliptic operator and the normal derivative of $\varphi$. We refer the reader to~\cite{Aubin1972} which provides a comprehensive overview of the abstract linear Neumann problems.\\
The minimax state estimation problem can be formulated as follows: given incomplete and noisy observations of the state, $ y(t) = C(t) \varphi + \eta(t)$, and a linear operator $V$, find a vector $\widehat{V\varphi}$, the minimax estimate of the vector $V\varphi$, such that
\begin{equation*}
\sup_{G,G_1} E\|V\varphi - \widehat{V\varphi}\|^2=\inf_{L(y)}\sup_{G,G_1} E\|V\varphi - L(y)\|^2
\end{equation*} where $L(y)$ is a linear functional of observations from a given class, and $\eta$ denotes a realisation of a random process with values in some Hilbert space $H$ such that $E\eta(t)=0$ and the correlation operator, $R_\eta$ is unknown and belongs to the given bounding set $G_1$. In other words, the minimax estimate $\widehat{V\varphi}$ of the vector $V\varphi$ has minimal mean-squared worst-case estimation error, corresponding to the worst-case realisation of the tuple $(f_0,f_1,R_\eta)\in G\times G_1$, in some class of linear estimates $L(y)$. This rather intuitive description is made precise in~\cref{sec:problem-statement-1}.

Minimax (set-membership/worst-case) state estimation framework has a long history: many authors studied minimax estimates for ordinary differential equations~\cite{Bertsekas1971,Tempo1985,Kurzhanski1997,Milanese1991}, differential-algebraic equations~\cite{Zhuk2009c,Zhuk2012AMO, SZMP_Automatica17, SZMP_TAC16}, abstract linear equations~\cite{Zhuk2009d,SZAP_IFAC17,Nakonechnii1978} assuming fully deterministic uncertainty description for both uncertain model errors and observation noises. Note that, for linear finite-dimensional dynamics and stochastic uncertainty description, the minimax estimate is equivalent to that of the Kalman filter~\cite{KrenerMinimax}.

In contrast, in this paper we study the case of mixed uncertainty description, i.e. deterministic model errors and stochastic measurement noise, and parameter-dependent observations. Our key contributions are:
\begin{itemize}
\item exact analytical representations of the minimax estimate $\widehat{V\varphi}$ provided $G$ is a convex possibly unbounded set, and $G_1$ is bounded, and $(N,\delta)$ are associated with the abstract Neumann problem (\cref{t:unbG},\cref{t:vector_mest,p:scalar_mest_hatphi})
\item sufficient conditions for the worst-case estimation error to decay to $0$ for $t\to+\infty$ (\cref{c:unboundedG_asymptotic_convergence,p:convergence_boundedG})
\end{itemize}
To the best of our knowledge, the aforementioned results are new and have not appeared in the mathematical literature yet. Let us stress that apart from pure theoretical contribution, these results are of particular significance in practice. Indeed, depending on the interpretation of the parameter $t$, one may get quite interesting applications: for instance, if $t$ is understood as a number of an experiment, $C(t)\equiv K$ is a compact operator of infinite rank, e.g. Volterra operator, and $y(t)$ denotes the experiment's outcome which is corrupted by a realisation of the noise $\eta(t)$, then the minimax estimates $\widehat{V\varphi}(t)$, $\widehat{V\varphi}(t+1)\,,\dots$ approach the solution of the least-squares problem $\|K\varphi - Kx\|\to\min_x$ with the minimal norm. In other words, the minimax estimation in this specific case implements an iterative pseudoinversion\footnote{$K^+$ denotes the Moore-Penrose pseudoinverse of $K$.} procedure converging to $K^+K\varphi$. Note that, in this example, $t$ can be interpreted as a number of the iteration. The operator $N$ plays role of a pre-conditioner which can significantly speed up the convergence. We stress that the aforementioned procedure operates subject to the noise process $\eta$, hence the minimax estimation combines the outcomes $y(t)$ of different experiments so that the pseudoinversion is stitched together with de-noising, and provides an estimate of the ``true'' solution $\varphi$ which is robust to uncertain $f_0,f_1$ and $R_\eta$. If $K$ was of Volterra type, i.e. $K\varphi(x) = \int_0^x \varphi(z)dz$, then the minimax estimator would simply differentiate $y$, i.e. approximate $\partial_z y(z,t)$ given a number of noisy measurements of $y(\cdot,t)$. Differentiators of such type are of particular importance in the sliding mode control~\cite{ZhukPolyakov2015_chapter}. We refer the reader to~\cref{sec:numer-exper} for the further details on this.\\
On the other hand, the case of mixed uncertainty description, i.e. deterministic model errors and stochastic measurement noise is quite common in practise. For instance, engineers usually work with approximate physical models, i.e. Galerkin projections of PDEs~\cite{ZhukSISC13}, and have little or no information about statistics of the model errors $f_1$ and $f_0$, which usually absorb discretization errors and physical simplifications, hence it is reasonable to assume deterministic uncertainty descriptions for $f_0$ and $f_1$ in the form of a bounding set $G$. In contrast, the measuring device described by $C$ may have a rather good description of the error statistics, and hence it makes sense to assume that the noise in observations is random. Moreover, the assumption that the correlation operator of the noise is not known but belongs to the given bounding set $G_1$ makes the state estimation robust with respect to fluctuations in the second moments.

One of the key applications of the minimax estimation framework is in data assimilation, an important component in many modern industrial cyber-physical systems which improves the accuracy of forecasts provided by physical models by optimally combining their states -- \emph{a priori} knowledge encoded in equations of mathematical physics -- with \emph{a posteriori} information in the form of sensor data, the so-called state estimation, and evaluates forecast reliability by taking into account uncertainty, i.e., model errors or sensor noise. To name just a few, let us mention online data assimilation of GPS coordinates into a macroscopic traffic model~\cite{ZhukITS14,SZTTABTS_CDC16}, assimilation of water elevation levels into the Saint Venant equations for flood prediction~\cite{ZhukAdvWaterRes2016}, adaptive grids~\cite{SZHvanB_SISC16}, robust source estimation~\cite{SZTTRSSM_SISC16} and air quality monitoring~\cite{MalletZhukJGR13}. We refer the reader to~\cite{ReichDA2015,StuartDA2015} for the further discussion on modern data assimilation. Another notable application of minimax estimation is in sliding mode control~\cite{ZhukPolyakov2015_chapter,SZAP_TAC16}.

This paper is organized as follows. The next section presents a brief overview of abstract Dirichlet/Neumann problems. \cref{sec:problem-statement-1} presents the formal problem statement; the minimax state estimates are shown in section~\ref{sec:main_results}. Application to pseudoinversion together with the numerical experiment is described in section~\ref{sec:numer-exper}. Section~\ref{sec:conclusion} contains conclusions.\\

\textbf{Notation.} Given an abstract Hilbert space $H$ we denote by $(\cdot,\cdot)_H$ its canonical inner product with values in $R$, and set $\|x\|^2_H:=(x,x)$ for any $x\in H$. $[x,y]$ denotes an element of $X\times Y$, the Cartesian product of two Hilbert spaces $X$ and $Y$. We also define a space of all linear continuous operators from a Hilbert space $H_1$ to Hilbert space $H_2$ by $\mathscr{L}(H_1,H_2)$, $H_1^\star$ denotes the adjoint space of $H_1$, $\Lambda_{H_1}$ denotes the canonical isomorphism of the Hilbert space $H_1$ onto $H_1^\star$, $I$ denotes the identity operator. $\langle\cdot,\cdot\rangle$ denotes the duality pairing between $H$ and its adjoint space $H^\star$. 

\section{Abstract Dirichlet/Neumann problem}
\label{sec:euler}
Assume that $H_0$, $H_{-}$ and $H_+$ are given Hilbert spaces such that \begin{equation*}
H_+\subseteq H_0\subseteq H_{-}\,.
\end{equation*} In what follows, we identify $H_0$ with its adjoint space $H_0^\star$. Let $(\cdot,\cdot)_{0,-,+}$ denote inner products in $H_{0,-,+}$ respectively, and let $a$ be a continuous bilinear form on $H_+\times H_+$ such that
\begin{equation}
  \label{eq:coercive_form}
\exists\alpha>0:\quad a(\phi,\phi)\ge\alpha^2 (\phi,\phi)_+\,,\quad \forall \phi\in H_+\,,\alpha\ne0\,.
\end{equation}
Let $H_\partial$ denote a Hilbert space and consider a linear operator $\gamma\in\mathscr{L}(H_+,H_\partial)$ such that
\begin{itemize}
\item $\gamma(H_+)=H_\partial$
\item \(
\mathring{H}_+:=\operatorname{ker}(\gamma) = \{\phi\in H_+:\gamma\phi=0\}
\) is dense in $H_+$
\end{itemize}
Define $\mathring{H}_{-}:=\mathring{H}_+^\star$. Then, clearly, \begin{equation*}
\mathring{H}_+\subseteq H_0\subseteq \mathring{H}_{-}\,.
\end{equation*} Define a linear operator \begin{equation*}
\forall\phi\in H_+ :\quad N\phi:=a(\phi,\cdot)\in \mathring{H}_{-}\,,
\end{equation*} i.e. $N$ maps a vector $\phi\in \mathring{H}_+$ to a linear continuous functional $\psi\mapsto \ell_\phi = a(\phi,\psi)$ over $\mathring{H}_+$, so that $\ell_\phi\in \mathring{H}_{-} = \mathring{H}_+^\star$. The operator $N$ is a bounded linear operator in the Hilbert space \(
H_+(N):=\{\phi\in H_+: N\phi\in H_0\}
\) equipped with the graph norm $\|\phi\|_{H_+(N)}:=(\|\phi\|^2_++\|N\phi\|^2_0)^\frac 12$. It is not hard to see, by definition of $N$, that \begin{equation*}
N\in\mathscr{L}(H_+,\mathring{H}_{-})\cap  \mathscr{L}(H_+(N),H_{0})\,.
\end{equation*}
Define $N^+$, the formal adjoint of $N$ as follows: \begin{equation*}
\forall\psi\in H_+ :\quad N^+\psi:=a(\cdot,\psi)\in \mathring{H}_{-}\,.
\end{equation*}
Clearly, $N^+\in\mathscr{L}(H_+,\mathring{H}_{-})\cap\mathscr{L}(H_+(N^+), H_0)$ where $H_+(N^+):=\{\phi\in H_+: N^+\phi\in H_0\}$. If the form $a$ is symmetric, i.e. $a(\phi,\psi)=a(\psi,\phi)$, then $N=N^+$.

\emph{The abstract Dirichlet problem} associated with the form $a$ is to find $\varphi\in H_+(N)$ such that:
\begin{equation}
  \label{eq:dirichlet}
N\varphi = f\,, \quad\gamma\varphi = f_0\,, \quad f_0\in H_\partial\,, f\in H_0\,.
\end{equation}
We stress that there exist a linear bounded operator $\delta\in\mathscr{L}(H_+(N),H_\partial^\star)$ such that the following Green formula holds true:
\begin{equation}\label{eq:green1}
a(\phi,\psi) = (N\phi,\psi)_0+ \langle\delta\phi,\gamma\psi\rangle\,,\quad \phi\in H_+(N)\,, \psi\in H_+\,,
\end{equation}
and the formal adjoint of $N$, $N^+$ verifies the following equality:
\begin{equation}\label{eq:Nadjoint}
(N^+\psi, \phi)_0 - (\psi, N\phi)_0 = \langle \gamma\psi, \delta\phi\rangle- \langle\delta^+\psi,\gamma\phi\rangle\,, \phi\in H_+(N)\,,\psi\in H_+(N^+)\,,
\end{equation}
where $\delta^+\in \mathscr{L}(H_+,H^\star_\partial)$ is such that \begin{equation*}
a(\psi,\phi) = (N^+\phi,\psi)_0 + \langle\delta^+\phi,\gamma\psi\rangle\,,\quad \phi\in H_+\,, \psi\in H_+(N^+)\,.
\end{equation*}
\emph{The abstract Neumann problem} associated with the form $a$ is to find $\phi\in H_+(N)$ such that:
\begin{equation}
  \label{eq:neumann}
N\phi = f\,, \quad\delta\phi = f_0\,, \quad f_0\in H^\star_\partial\,, f\in H_0\,.
\end{equation}
This latter problem can be equivalently reformulated in a variational form~\cite{Aubin1972}, namely $\phi\in H_+(N)$ solves~\cref{eq:neumann} if and only if
\begin{equation}
  \label{eq:neumann_var}
a(\phi,\psi) = (f,\psi)_0+ \langle f_0,\gamma\psi\rangle \,, \quad \forall \psi\in H_+\,.
\end{equation}
We stress that either~\cref{eq:dirichlet} or~\cref{eq:neumann} have the unique solution provided~\cref{eq:coercive_form} holds true~\cite{Aubin1972}. The aforementioned abstract Dirichlet/Neumann problems incapsulate a wide class of mixed boundary value problems for linear elliptic equations over Lipchitz domains. A specific example of the abstract Neumann problem will be given below, in the section~\ref{sec:numer-exper}. We refer the reader to~\cite{Aubin1972} for the further details on abstract Dirichlet/Neumann problems.

\section{Problem statement}
\label{sec:problem-statement-1}
Assume that we observe a vector-function $y(t)$ with values in a Hilbert space $H$ such that
\begin{equation}
  \label{eq:obs}
  y(t) = C(t) \varphi + \eta(t)\,,
\end{equation}
where
\begin{itemize}
\item $t\mapsto C(t)\in\mathscr{L}(H_0,H)$ is a given linear transformation, an abstract model of a measuring device
\item $\eta(t)$ denotes a realisation of a random process with values in $H$ such that $E\eta(t)=0$ and $\int_0^T E\|\eta(t)\|^2_H dt<+\infty$, i.e. the process has zero mean and finite second moments; moreover, the correlation operator $(R_\eta(t,s)x_1,x_2)_{H}:=E(\eta(t),x_1)_H(\eta(s),x_2)_{H}$ is unknown and belongs to the given bounding set $G_1$

\item $\varphi\in H_+(N)$ solves\footnote{As noted above,~\cref{eq:coercive_form} implies that~\cref{eq:state} has the unique solution $\varphi$ for any $f_1$ and $f_0$.} the following abstract Neumann problem (in the operator form)
\begin{equation}
  \label{eq:state}
  N\varphi = B_1 f_1\,\quad \delta\varphi=B_0f_0\,,
\end{equation}
subject to uncertain disturbances $f_{0,1}\in F_{0,1}$ from a given bounding set $G$, and $B_0\in\mathscr{L}(F_0,H_\partial^\star )$ and $B_1\in\mathscr{L}(F_1,H_0)$ are given linear operators.
\end{itemize}
Let us now introduce the notion of the minimax estimate. Given $y(t)$ defined as in~\cref{eq:obs}, a vector $c\in H_V$ and linear bounded operators $V\in \mathscr{L}(H_0,H_V)$ and $U(t)\in\mathscr{L}(H,H_V)$ we say that an affine functional $y(\cdot)\mapsto U(y)(t):=\int_0^T U(t) y(t) dt + c$ is an estimate of $V\varphi$. Now, the minimax estimate of $V\varphi$ is defined as follows:
\begin{definition}\label{d:mest}
Given $U\in L^2(0,T,\mathscr{L}(H,H_V))$ and $c\in H_V$ we define the worst-case estimation error $\sigma(U,c)$ associated with $U$ as follows: \begin{equation*}
\sigma^2(U,c):= \sup_{[f_0,f_1]\in G, R_\eta\in G_1} E\| V\varphi- \int_0^T U(t) y(t) dt-c\|_{H_V}^2\,.
\end{equation*}
$\widehat{V\varphi}=\int_0^T \hat U(t) y(t) dt+\hat c$ is said to be a \emph{minimax estimate} of $V\varphi$ provided \begin{equation*}
\inf_{c,U\in L^2(0,T,\mathscr{L}(H))} \sigma(U,c) = \sigma(\hat U,\hat c):=\hat\sigma\,.
\end{equation*}
The number $\hat\sigma$ is called the \emph{minimax error}.
\end{definition}
Note that the minimax estimate $\widehat{V\varphi}$ of $V\varphi$, defined according to~\cref{d:mest}, has the following property: given an operator-valued function $t\mapsto U(t)$ we can compute the worst-case estimation error $\sigma(U)$, associated to the corresponding estimate $U(y)=\int_0^T U(t) y(t) dt + c$, by evaluating the mean-squared distance $E\| V\varphi- \int_0^T U(t) y(t) dt\|_{H_V}^2$ between all values of $U(y)$ and $V\varphi$ generated when $(f_0,f_1)$ runs through $G$, and the correlation function of the noise process runs through $G_1$; the minimax estimate is then defined as a vector $\widehat{V\varphi}$ having the minimal worst-case error $\hat\sigma$. \textbf{The aim of this paper} is to study existence, uniqueness and various representations of the minimax estimates.

\section{Main results}
\label{sec:main_results}
This section presents existence and uniqueness theorems for the minimax estimates for the case of generic convex bounding sets. The case of ellipsoidal bounding sets is considered in details and the corresponding minimax estimates are defined as the unique solution of an abstract boundary-value problem. Finally, asymptotic behaviour of the estimates is studied, namely sufficient conditions for the estimation error to converge to $0$ are formulated.

\subsection{Existence and uniqueness of the minimax estimates}
\label{sec:exist-uniq-minim}

In this section we present existence and uniqueness results. Take $\ell\in H_V$ and let $z\in H_+(N^+)$ denote the unique solution of the following adjoint equation in the variational form:
\begin{equation}
  \label{eq:adjoint_var}
a(\psi,z) = ( \ell,V\psi)_{H_V} - \int_0^T (\ell,U(t)C(t)\psi)_{H_V}  dt \,, \quad \forall \psi\in H_+(N)\,.
\end{equation}
Recall that we identified $H_0$ with $H^\star_0$, and so $V^\star\in \mathscr L(H_V^\star, H_0)$. Clearly, $z$ solves the equivalent operator equation:
\begin{equation}
  \label{eq:adjoint_op}
N^+ z = V^\star \Lambda \ell - \int_0^T C^\star(t) U^\star(t) \Lambda\ell dt\,,\quad \delta^+z=0\,,
\end{equation}
where $\Lambda$ denotes the canonical isomorphism of $H_V$ onto $H_V^\star$. Define \begin{equation*}
\alpha(U,f_0,f_1,\ell):= (B_1 f_1,z)_{0} +\langle B_0f_0,\gamma z\rangle\,, \mathscr D_\alpha:=\{U: \sup_{\ell:\|\ell\|_{H_V}=1, [f_0,f_1]\in G_0}\alpha(U,f_0,f_1,\ell)<+\infty\}
\end{equation*} The following lemma provides means for computing the worst-case error defined in~\cref{d:mest}.
\begin{proposition} \label{p:1}
Let $\varphi$ solve~\cref{eq:neumann_var} and $z$ solve~\cref{eq:adjoint_var} for some  $U\in L^2(0,T,\mathscr{L}(H,H_V))$ and $c\in H_V$. Let us further assume that $G_1$ is bounded, and $G=[\bar f_0,\bar f_1] + G_0$ for a given $[\bar f_0,\bar f_1]\in F_0\times F_1$, and a convex set $G_0$ such that $G_0=(-1)G_0$. If $U\in \mathscr{D}_\alpha$ then
\begin{equation}
  \label{eq:sigmaUc}
  \begin{split}
    \sigma^2(U,c) &= \sup_{\ell:\|\ell\|_{H_V}=1} \sup_{[f_0,f_1]\in G_0} \left(\alpha(U,f_0,f_1,\ell) + |(\ell,c)_{H_V} - \alpha(U,\bar f_0,\bar f_1,\ell)|\right)^2\\
& + \sup_{G_1} E \|\int_0^T U(t)\eta(t) dt \|^2_{H_V}
\end{split}
\end{equation}
and $\sigma^2(U,c)=+\infty$ otherwise.
\end{proposition}
\begin{proof}
Recall that $E\eta(t)=0$ and set by definition $q:=V\varphi- \int_0^T U(t) C(t) dt\varphi -c$. Since \begin{equation*}
\|V\varphi- \int_0^T U(t) y(t) dt-c\|_{H_V}=\|q\|^2_{H_V}+2(q,\int_0^T U(t) \eta(t) dt)_{H_V}+
\|\int_0^T U(t) \eta(t) dt\|^2_{H_V}
\end{equation*}
and $E(q,\int_0^T U(t) \eta(t) dt)_{H_V}=0$ it follows that \begin{equation*}
E\|V\varphi - \int_0^T U(t) y(t) dt-c\|^2_{H_V} = \|q\|^2_{H_V} + E \|\int_0^T U(t) \eta(t) dt\|^2_{H_V}\,.
\end{equation*}
We claim that
\begin{equation}
  \label{eq:error}
    E\|V\varphi - \int_0^T U(t) y(t) dt-c\|^2_{H_V} = \sup_{\ell:\|\ell\|_{H_V}=1} \left(\alpha(U,f_0,f_1,\ell)- (\ell,c)_{H_V}\right)^2+ E \|\int_0^T U(t)\eta(t) dt \|^2_{H_V}
\end{equation}
Indeed, note that, by definition of the norm of a Hilbert space one has:
\begin{equation*}
  \begin{split}
    \|q\|^2_{H_V} &= \sup_{\ell:\|\ell\|_{H_V}} \left((\ell, V\varphi)_{H_V} - (\ell, \int_0^T  U(t) C(t) dt\varphi)_{H_V} - (\ell,c)_{H_V}\right)^2 \\
    &\stackrel{\text{by~\cref{eq:adjoint_var} with } \psi=\phi}{=} \sup_{\ell:\|\ell\|_{H_V}=1} \left(a(\varphi,z)  - (\ell,c)_{H_V}\right)^2\\
& \stackrel{\text{by~\cref{eq:neumann_var} with } \phi=z}{=} \sup_{\ell:\|\ell\|_{H_V}=1} \left((Bf_1,z)_0+ \langle B_0f_0,\gamma z\rangle  - (\ell,c)_{H_V}\right)^2
  \end{split}
\end{equation*}
This proves~\cref{eq:error}. To prove~\cref{eq:sigmaUc} take $U\in \mathscr{D}_\alpha$. Then, by~\cref{eq:error} and by~\cref{d:mest} it follows that:\begin{equation*}
\sigma^2(U,c)= \sup_{\ell:\|\ell\|_{H_V}=1} \sup_{[f_0,f_1]\in G_0} \left(
\alpha(U,f_0,f_1,\ell) - ((\ell,c)_{H_V} - \alpha(U,\bar f_0,\bar f_1,\ell))\right)^2 + \sup_{G_1} E \|\int_0^T U(t)\eta(t) dt \|^2_{H_V}\,.
\end{equation*}
Since $G_0=(-1) G_0$ we can write: \begin{equation*}
\alpha(U,f_0,f_1,\ell)\in [-\sup_{[f_0,f_1]\in G_0} \alpha(U,f_0,f_1,\ell), \sup_{[f_0,f_1]\in G_0} \alpha(U,f_0,f_1,\ell) ]
\end{equation*} hence, by a simple equality $\sup_{k\in [-a,a]} | k - m| = a + |m|$ we derive~\cref{eq:sigmaUc}.
Finally, we note that, by definition of $\mathscr D_\alpha$ and by~\cref{eq:sigmaUc} it follows that $\sigma^2(U,c)=+\infty$ if $U\not\in \mathscr{D}_\alpha$.
\end{proof}
Define \begin{equation*}
\Phi(U):=  \sup_{\ell:\|\ell\|_{H_V}=1} \sup_{[f_0,f_1]\in G_0} \alpha^2(U,f_0,f_1,\ell) + \sup_{G_1} E \|\int_0^T U(t)\eta(t) dt \|^2_{H_V}\,.
\end{equation*}
\begin{corollary}\label{c:hatc}
Let $\hat U\in\operatorname{Argmin}\Phi(U)$ and define
\begin{align}
    &N\bar\varphi = B_1 f_1\,, \quad \delta\bar\varphi = B_0f_0\,,  \label{eq:fibar}\\
    &\hat c:= V\bar\varphi - \int_0^T  \hat U(t)C(t)dt \bar\varphi\,.   \label{eq:chat}
\end{align}
It then follows that $\min_{U,c}\sigma^2(U,c)=\sigma^2(\hat U,\hat c)$.
\end{corollary}
\begin{proof}
Clearly $\sigma^2(U,c)\ge\Phi(U) \ge\Phi(\hat U)$. Thus $\Phi(\hat U)= +\infty$ implies that $\sigma(U,c)=+\infty$ and the statement of~\cref{c:hatc} holds true. If $\Phi(\hat U)< +\infty$ then, by~\cref{eq:sigmaUc}, we have:
\begin{equation*}
(\ell,\hat c)_{H_V}\stackrel{\text{by~\cref{eq:fibar,eq:adjoint_var}}}{=}a(\bar\varphi,z)\stackrel{\text{by~\cref{eq:chat}}}{=}(B_1\bar f_1,z)_0+\langle B_0\bar f_0, \gamma z\rangle
\end{equation*}
Hence, for any $c\in H_V$ and $U\in L^2(0,T,\mathscr{L}(H,H_V))$ it holds that $\sigma^2(U,c)\ge\Phi(U) \ge\Phi(\hat U)=\sigma^2(\hat U,\hat c)$. This completes the proof.
\end{proof}
\begin{remark}\label{rem:barf0}
Note that as per~\cref{c:hatc} we can restrict our attention to the case $\bar f_1=0$, $\bar f_0=0$ without loss of generality. The interpretation of this fact is quite obvious: since~\cref{eq:state} is linear it follows that one can shift the pre-image of $G$ w.r.t. $N,\delta$ by $\bar\varphi$, and then estimate the deviation of the ``true'' solution $\varphi$ from $\bar\varphi$. The latter is achieved by working with $G_0$ instead of $G$ or, equivalently, by assuming that $\bar f_1=0$, $\bar f_0=0$ and setting $\hat c=0$.
\end{remark}
\subsubsection{Unbounded disturbances}
\label{sec:unbo-dist}

The following theorem provides a representation of the minimax estimate for the case of no information available on $f_0$ and $f_1$, i.e. $G_0=F_0\times F_1$, and an ellipsoidal bounding set for the correlation operator $R_\eta$. Assume that $Q(t)\in\mathscr L(H)$ and $Q(t)=Q^\star(t)\ge\alpha^2 I$, $\alpha\ne0$, and $t\mapsto \|Q(t)\|\in C(0,T)$.
\begin{theorem}\label{t:unbG}
Assume that $G_0=F_0\times F_1$. Then
\begin{equation}
  \label{eq:sigma=supG2}
+\infty>\sigma^2(U,0)=\sup_{G_1}E\|\int_0^T U(t)\eta(t)dt\|^2_{H_V} \Leftrightarrow U\in \mathscr D_1:=\{U:(z,B_1f_1)_0=0\,,\quad \langle B_0f_0,\gamma z\rangle=0\}
\end{equation}
If $\mathscr D_1\ne\varnothing$ and, in addition,
\begin{equation}
  \label{eq:G1}
  G_1=\{R_\eta:\int_0^T E (Q(t) \eta(t),\eta(t))_H dt \le \gamma_T^2\}
\end{equation}
then
\begin{equation}
  \label{eq:minsigmaG1}
\min_U\sigma^2(U,0)=\gamma^2_T\sup_{\ell:\|\ell\|_{V}=1}(\ell,Vp)_{H_V}\,,\quad \widehat{V\varphi}=\int_0^T \hat U(t) y(t)dt = V\hat\varphi
\end{equation}
provided $\hat\varphi$ and $p$ solve the following system of equations:
\begin{equation}
  \label{eq:hatp}
  \begin{split}
  N^+z &= V^\star\Lambda\ell - B_T p\,, \quad \delta^+ z=0\,, B_T = \int_0^T C^\star(t)\Lambda_H Q(t) C(t) dt\,,\\
Np &= B_1\lambda_1\,, \delta p = B_0\lambda_0\,, \quad \langle \gamma z, B_0 f_0\rangle = 0\,, (B_1 f_1, z)_0=0\,, \forall f_0\in F_0\,, f_1\in F_1\,.
  \end{split}
\end{equation}
and
\begin{equation}
  \label{eq:hatphi}
  \begin{split}
N^+\hat p &= \int_0^T C^\star \Lambda_H Q(t) y(t) dt - B_T\hat\varphi\,, \quad\delta^+ \hat p = 0\,, \\
N\hat\varphi &= B_1\mu_1\,, \delta\hat\varphi = B_0\mu_0\,,\quad \langle B_0 f_0,\gamma\hat p\rangle=0\,, (B_1 f_1, \hat p)_0=0\,, \forall f_0\in F_0\,, f_1\in F_1\,.
  \end{split}
\end{equation}
\end{theorem}
\begin{proof}
  If $U\not\in \mathscr D_1$ then $\alpha(U,f_0,f_1,\ell)>0$ for some $[f_0,f_1]\in G_0$, and hence $\sup_{G_0}\alpha(U,f_0,f_1,\ell)=\sup_{G_0}\left((B_1 f_1,z)_{0} +\langle B_0f_0,\gamma z\rangle\right)=+\infty$. The latter implies that $\sigma^2(U,c)=+\infty$ by~\cref{eq:sigmaUc}. Clearly, $\sigma^2(U,c)=+\infty$ if $\mathscr D_1=\varnothing$. If $\mathscr D_1\ne\varnothing$ and $U\in \mathscr D_1$ we have that $\alpha(U,f_0,f_1,\ell)=0$ and so~\cref{eq:sigma=supG2} holds true.

To prove~\cref{eq:minsigmaG1} let us note that $\sup$ in~\cref{eq:sigma=supG2} is attained at a ``deterministic process'', i.e.
\begin{equation}
  \label{eq:G1G2}
\sup_{G_1}E\|\int_0^T U(t)\eta(t)dt\|^2_{H_V} = \sup_{G_2}\|\int_0^T U(t)g(t)dt\|^2_{H_V}\,,\quad G_2:=\{g: \int_0^T (Q(t)g(t),g(t))_{H_V} dt \}
\end{equation}
Indeed, since for any $\eta\in L^2(0,T,H)$ it holds
\begin{equation}\label{eq:supUeta}
\|\int_0^T U(t)\eta(t)dt\|^2_{H_V}\le \sup_{\ell:\|\ell\|_{H_V}=1} \int_0^T (Q^{-1}\Lambda^{-1} U^\star\Lambda\ell, \Lambda^{-1} U^\star\Lambda\ell)_{H_V}dt \int_0^T (Q(t)\eta(t),\eta(t))_{H_V} dt
\end{equation}
we find that
\begin{equation}
  \label{eq:supG1}
\sup_{G_1}\|\int_0^T U(t)g(t)dt\|^2_{H_V} = \gamma_T^2 \sup_{\ell:\|\ell\|_{H_V}=1} \int_0^T (Q^{-1}\Lambda^{-1} U^\star\Lambda\ell, \Lambda^{-1} U^\star\Lambda\ell)_{H_V}dt\,.
\end{equation}
On the other hand, $E \int_0^T (Q(t)\eta(t),\eta(t))_{H_V} dt = \int_0^T E(Q(t)\eta(t),\eta(t))_{H_V} dt\le \gamma_T^2$, hence, by~\cref{eq:supUeta}, we find that \begin{equation*}
\sup_{G_1}E\|\int_0^T U(t)\eta(t)dt\|^2_{H_V}\le \sup_{G_2}\|\int_0^T U(t)g(t)dt\|^2_{H_V}\,.
\end{equation*}
Now, let $g_k$ denote a sequence of $H_V$-valued functions such that \begin{equation*}
\lim_{k\to\infty} \|\int_0^T U(t)g_k(t)dt\|^2_{H_V} = \sup_{G_2}\|\int_0^T U(t)g(t)dt\|^2_{H_V}\,,
\end{equation*}
and let $\xi$ denote a scalar random variable such that $E\xi=0$ and $E\xi^2=1$. Define $\eta_k(t):=\xi g_k(t)$. Then the correlation operator of $\eta_k$ takes the following form: $(R_k(t,s)x_1,x_2)_H=E(\eta_k(t),x_1)_H (\eta(s),x_2)_H = (g_k(t),x_1)_H(g_k(s),x_2)_H=g_k(t)\otimes g_k(s)$ and \begin{equation*}
\sup_{R_k\in G_1}E\|\int_0^T U(t)\eta_k(t)dt\|^2_{H_V}\le \sup_{G_1}E\|\int_0^T U(t)\eta(t)dt\|^2_{H_V}\le \lim_{k\to\infty} \|\int_0^T U(t)g_k(t)dt\|^2_{H_V} = \sup_{G_2}\|\int_0^T U(t)g(t)dt\|^2_{H_V}
\end{equation*} We get~\cref{eq:G1G2} by noting that $E\|\int_0^T U(t)\eta_k(t)dt\|^2_{H_V} = \|\int_0^T U(t)g_k(t)dt\|^2_{H_V}$.

Let us first prove~\cref{eq:minsigmaG1} for the following special case: $H_V=R$ and $V\varphi = (v,\varphi)_0$, $v\in H_0$. Note that $\Lambda=1$ in this case, $U\in L^2(0,T,H^\star)$ and $V^\star\ell=v$ provided $\|\ell\|_{V_H}=1$. Clearly, $U^\star=u$ for some $u\in L^2(0,T,H)$ such that $\Lambda_H u = U^\star$, $\Lambda_H$ denotes the canonical isomorphism of $H$ onto $H^\star$. By~\cref{eq:G1G2,eq:supUeta} it follows that $\sigma^2(u,0)= \gamma_T^2 \int_0^T (Q^{-1}u, u)_H dt$ is a strictly convex weakly lower semi-continuous functional, and $\lim_{\|u\|_{L^2(0,T,H)}\to+\infty}\sigma^2(u,0)=+\infty$. Hence, $\sigma^2$ has a unique minimum point $\hat u$. To find $\hat u$ recall the definition of $\mathscr D_1$ and define the Lagrange functional \begin{equation*}
\mathcal L(u,\lambda_0,\lambda_1):=\frac12\int_0^T (Q^{-1}u, u)_H dt + (B_1\lambda_1,z)_0 + \langle B_0\lambda_0,\gamma z\rangle\,,
\end{equation*}
provided $z$ solves~\cref{eq:adjoint_op}. The existence of $\hat u$ implies that there exist $\lambda_{0,1}$ such that $\lim_{\tau\to0+}\frac{d\mathcal L}{d\tau}(\hat u+\tau u_1,\lambda_0,\lambda_1)=0$ for any $u_1\in L^2(0,T,H)$. On the other hand
\begin{equation}
  \label{eq:Gateaux_trasformation}
\begin{split}
  &\lim_{\tau\to0+}\frac{d\mathcal L}{d\tau}(\hat u+\tau u_1,\lambda_0,\lambda_1)= \int_0^T (Q^{-1}\hat u, u_1)_H dt + (B_1\lambda_1,z_1)_0 + \langle B_0\lambda_0,\gamma z_1\rangle\\
  & N^+ z_1 = -\int_0^T C^\star \Lambda_H u_1 dt \,, \quad \delta z_1 =0\,.
\end{split}
\end{equation}
Define $p$ as follows: $Np=B_1\lambda_1$, $\delta p = B_0\lambda_0$. Then, by Green's identity~\cref{eq:Nadjoint}, we get: $\lim_{\tau\to0+}\frac{d\mathcal L}{d\tau}(\hat u+\tau u_1,\lambda_0,\lambda_1)= \int_0^T (Q^{-1}\hat u-C(t)p, u_1)_H dt = 0$ for any $u_1\in L^2(0,T,H)$. Hence, $\hat u = Q(t) C(t) p$ and $[p,z]$ solve~\cref{eq:hatp}. To prove~\cref{eq:minsigmaG1} we note that
\begin{equation}
  \label{eq:minsigma_v}
\sigma^2(\hat u,0)= \gamma_T^2 \int_0^T (Q^{-1}\hat u, \hat u)_H dt \stackrel{\text{by~\cref{eq:hatp}}}{=} \gamma_T^2 (B_T p,p)_0 = \gamma_T^2 \left((v,p)_0 - (N^+z,p)_0\right)  \stackrel{\text{by~\cref{eq:Nadjoint,eq:hatp}}}{=} \gamma_T^2 (v,p)_0
\end{equation}
which proves~\eqref{eq:minsigmaG1}. Finally, let us prove that $\widehat{V\varphi} = V\hat\varphi$. Indeed, \cref{eq:hatphi} is a system of the same type as~\cref{eq:hatp} and so it has the unique solution. We have:
\begin{equation*}
  \begin{split}
    \widehat{V\varphi} &= \int_0^T (\hat u(t),y(t))_H dt = \int_0^T (Q(t)C(t)p(t), y(t)) dt = (p, \int_0^T C^\star \Lambda_H Q(t)  y(t)_H dt\\
&= (N^+\hat p + B_T\hat\varphi, p)_0  = (B_T p,\hat\varphi)_0 = (v- N^+z,\hat\varphi)_0 = (v, \hat\varphi)_0
  \end{split}
\end{equation*}
where we used that $(N^+\hat p,p)_0=0$ and $(N^+z,\hat\varphi)_0=0$ by~\cref{eq:Nadjoint,eq:hatp,eq:hatphi}.

Let us now consider the general case. Since
\begin{equation}
  \label{eq:supG1_Voperator}
  \begin{split}
    \sup_{G_1} |\int_0^T U(t) g(t) dt \|^2_{H_V} dt &= \sup_{G_1} \sup_{\|\ell\|_{H_V} = 1} \left(\int_0^T (\Lambda^{-1} U^\star\Lambda \ell, g(t))_{H}dt\right)^2\\
    &= \gamma^2_T\sup_{\|\ell\|_{H_V} = 1}\int_0^T (Q^{-1}\Lambda^{-1} U^\star\Lambda \ell,\Lambda^{-1} U^\star\Lambda \ell)_H dt
  \end{split}
\end{equation}
it follows that
\begin{equation}
  \label{eq:infU_lowbnd}
  \begin{split}
\inf_{U\in\mathscr D_1}\sigma^2(U,0)&\ge \gamma^2_T \sup_{\|\ell\|_{H_V} = 1} \inf_{\tilde u}\{\int_0^T (Q^{-1}\tilde u,\tilde u)_H dt\,, \quad \tilde u=\Lambda^{-1}U^\star\Lambda\ell\,,\quad U\in\mathscr D_1\}\\
&\ge \gamma^2_T \sup_{\|\ell\|_{H_V} = 1} \inf_{u\in L^2(0,T,H)}\int_0^T (Q^{-1} u(t),u(t))_H dt \stackrel{\text{by~\cref{eq:minsigma_v}}}{=} \gamma^2_T \sup_{\|\ell\|_{H_V} = 1} (\ell,Vp)_0
\end{split}
\end{equation}
On the other hand, \(
\gamma^2_T (\ell,Vp)_{H_V} = \sup_{G_2}((\ell,V\varphi)_{H_V} - (\ell,V\hat\varphi)_{H_V})^2
\) and so \begin{equation*}
\gamma^2_T \sup_{\|\ell\|_{H_V} = 1} (\ell,Vp)_{H_V} = \sup_{G_2} \| V\varphi - V\hat\varphi\|^2_{H_V} = \sup_{G_1} E \| V\varphi - V\hat\varphi\|^2_{H_V}
\end{equation*}
Hence, $\sup_{G_1} E \| V\varphi - \widehat{V\varphi}\|^2_{H_V}\ge \sup_{G_1} E \| V\varphi - V\hat\varphi\|^2_{H_V}$, and hence $\widehat{V\varphi}= V\hat\varphi$. This proves~\eqref{eq:minsigmaG1}.
\end{proof}
The next corollary shows that the minimax error approaches $0$ asymptotically provided the observations are available and dominate the noise.
\begin{corollary}\label{c:unboundedG_asymptotic_convergence}
If all the assumptions of~\cref{t:unbG} are true, and, in addition, \begin{equation*}
\lim_{T\to\infty} \frac{\lambda_{min}(B_T)}{\gamma^{2}_T}=+\infty\,,\quad \lambda_{min}(B_T):=\inf_{\psi:\|\psi\|_H=1} (B_T\psi,\psi)_H\,,
\end{equation*} then $\lim_{T\to\infty}\sigma^2(\hat U,0)=0$.
\end{corollary}
\begin{proof}
It follows by~\cref{eq:minsigmaG1,eq:hatp} that $\min_U\sigma^2(U,0)=\gamma^2_T\sup_{\ell:\|\ell\|_{V}=1}(\ell,Vp)_{H_V}=\gamma^2_T\sup_{\ell:\|\ell\|_{V}=1}(B_Tp,p)_0$. Since $(B_Tp,p)_0\ge \lambda_{min} (B_T)\|p\|_0^2$ it follows that $\sup_{\ell:\|\ell\|_{V}=1}(\ell,Vp)_{H_V}\ge \lambda_{min} (B_T)\sup_{\ell:\|\ell\|_{V}=1}\|p\|_0^2$. On the other hand, $(\ell,Vp)_{H_V}\le \|\ell\|_{H_V} \sup_{q:\|q\|_0=1}\|Vq\|_{H_V}\|p\|_0$. Hence, we get: \begin{equation*}
\sup_{\ell:\|\ell\|_{H_V}=1}\|p\|_0^2\le \lambda^{-1}_{min} (B_T)\sup_{q:\|q\|_0=1}\|Vq\|_{H_V} \sup_{\ell:\|\ell\|_{H_V}=1}\|p\|_0\,
\end{equation*} As a result \begin{equation*}
\min_U\sigma^2(U,0)=\gamma^2_T\sup_{\ell:\|\ell\|_{V}=1}(\ell,Vp)_{H_V} \le \sup_{q:\|q\|_0=1}\|Vq\|^2_{H_V} \lambda^{-1}_{min} (B_T) \gamma^2_T\to 0\,, \quad T\to+\infty
\end{equation*}
\end{proof}
\begin{remark}
  \cref{c:unboundedG_asymptotic_convergence} implies that  $E\widehat{V\varphi}=V\varphi$, i.e. the minimax estimate is non-biased. In addition, if the operator $B_T$ has a discrete spectrum then $\lambda_{min}(B_T)$ equals to the minimal eigen value of $B_T$.
\end{remark}

\subsubsection{Bounded disturbances}
\label{sec:bounded-disturbances}

Assume that $G_0$ is an ellipsoid, i.e.
\begin{equation}
  \label{eq:ellipsoidG0}
G_0=\{[f_0,f_1]:(Q_0f_0,f_0)_{F_0}+(Q_1 f_1, f_1)_{F_1} \le 1\}\,,\quad Q_i\in \mathscr L(F_i)\,, Q_i\ge \beta_i I\,, \beta_i>0\,, i=1,2\,.
\end{equation}
\begin{theorem}\label{t:vector_mest}
Let $G_1$ be defined by~\cref{eq:G1} and assume that $R_\eta\in G_1$. Then
\begin{equation}
  \label{eq:Lpinf}
\inf_{U,0}\sup_{G_0,G_1} E\|V\varphi-\widehat{V\varphi}\|_{H_V}^2\ge \sup_{\ell:\|\ell\|_{H_V}=1}(\ell,Vp)_{H_V}
\end{equation}
provided $p$ solves
\begin{equation}
  \label{eq:zp}
  \begin{split}
    N^+z &= V^\star\Lambda\ell - \gamma_T^{-2} B_Tp\,,\quad \delta^+z=0\,,\\
    Np &= B_1 Q_1^{-1}\Lambda_1 B_1^\star z\,, \quad \delta p = B_0 Q_0^{-1} \Lambda_0 B_0^\star\gamma z
  \end{split}
\end{equation}
where $\Lambda_i$ is the canonical isomorphism of $F_i$ onto it's adjoint.
\end{theorem}
\begin{proof}
Let us recall Schwartz inequality: $(B_1x,y)_0\le (B_1 Q^{-1}\Lambda_1 B_1^\star x,x)_0^\frac12(Qy,y)_0^\frac12$. We also recall~\cref{rem:barf0}, namely that $\bar f_{0,1}=0$. Let us set $u(\ell):=\Lambda^{-1} U^\star\Lambda \ell$ for some $U\in L^2(0,T,\mathscr{L}(H,H_V))$ and $\ell_2\in H_V$. Now, by~\cref{eq:sigmaUc,eq:G1G2} we get that:
\begin{equation}\label{eq:sigma_derivation}
  \begin{split}
    \sigma^2(U,0) &= \sup_{\|\ell\|_{H_V}=1}\sup_{G_0} \left( (z,B_1f_1)_0+\langle \gamma z, B_0 f_0 \rangle \right)^2 + \sup_{\|\ell_2\|_{H_V} =1}\sup_{G_2} \left(\int_0^T \langle U^\star\Lambda\ell_2,g\rangle dt\right)^2\\
&\stackrel{\text{by Schwartz ineq.}}{=}\sup_{\|\ell\|_{H_V}=1}\left(( B_1 Q_1^{-1}\Lambda_1 B_1^\star z,z)_0 + \langle B_0 Q_0^{-1} \Lambda_0 B_0^\star\gamma z,\gamma z\rangle \right)^2\\
&\qquad\qquad+ \gamma^2_T\sup_{\|\ell_2\|_{H_V} =1}\left(\int_0^T (Q^{-1}u(\ell_2),u(\ell_2))_{H_V} dt\right)^2\\
&\ge\sup_{\|\ell\|_{H_V}=1} \Phi(u(\ell))\,, \Phi(u(\ell)):=( B_1 Q_1^{-1}\Lambda_1 B_1^\star z,z)_0 + \langle B_0 Q_0^{-1} \Lambda_0 B_0^\star\gamma z,\gamma z\rangle +\int_0^T (Q^{-1}u(\ell),u(\ell))_{H_V} dt\,,\\
&\ge \inf_{U\in L^2(0,T,\mathscr L(H,H_V))} \sup_{\|\ell\|_{H_V}=1}\Phi(u(\ell))\ge \sup_{\|\ell\|_{H_V}=1} \inf_{U\in L^2(0,T,\mathscr L(H,H_V))} \Phi(u(\ell))\ge \inf_{u\in L^2(0,T,H)} \Phi(u)
  \end{split}
\end{equation}
Using the same type of argument~\footnote{See the proof of~\cref{eq:minsigmaG1} for the special case: $H_V=R$ and $V\varphi = (v,\varphi)_0$, $v\in H_0$.} as in the proof of~\cref{t:unbG} it is not hard to show that there exists the unique $\hat u$ such that $\inf_{u} \Phi(u) = \Phi(\hat u)$. Then $\lim_{\tau\to0+}\frac{d\Phi}{d\tau}(\hat u+\tau u_1)=0$ for all $u_1\in L^2(0,T,H)$. By using~\cref{eq:zp,eq:Nadjoint} and an argument similar to~\cref{eq:Gateaux_trasformation} we find that $\hat u = \gamma_T^{-2} Q(t) C(t) p$ and $\Phi(\hat u) = (\ell,Vp)_0$. Thus, $\inf_{U}\sigma^2(U,0) \ge \sup_{\|\ell\|_{H_V}=1} \inf_U \Phi(u(\ell))\ge \sup_{\|\ell\|_{H_V}=1}\Phi(\hat u) = \sup_{\|\ell\|_{H_V}=1} (\ell,Vp)_{H_V}$. This completes the proof.
\end{proof}
In fact, \cref{t:vector_mest} shows that the minimax error of the minimax estimate $\widehat{V\varphi}$ of the vector $V\varphi$ is bounded from below by the induced norm of the linear operator $\ell\mapsto Vp(\ell)$. The following corollary shows that this lower bound is exact for a specific type of $V$.
\begin{corollary}\label{c:scalar_mest}
  If $H_V=R^1$ and $V\varphi = (v,\varphi)_0$ for some $v\in H_0$ then $\inf_{U}\sigma^2(U,0)=(v,p)_0$, provided $p$ solves~\cref{eq:zp} with $V^\star\Lambda\ell=v$.
\end{corollary}
\begin{proof}
  Indeed, $U\in L^2(0,T,H^\star)$ and $\Lambda=1$ so that all $\sup_{\|\ell\|_{H_V}=1}\{\cdot\}$ in~\cref{eq:sigma_derivation} can be dropped, and we get that $u(\ell)=u:=U^\star\in L^2(0,T,H)$. Hence $\sigma^2(U,0) = \Phi(u)$, and $\inf_{U}\sigma^2(U,0)=(\ell,Vp)_{H_V}=(v,p)_0$.
\end{proof}
\begin{proposition}\label{p:scalar_mest_hatphi}
Assume that $H_V=R^1$ and $V\varphi = (v,\varphi)_0$ for some $v\in H_0$, and let $[\hat p,\hat\varphi]$ solve the following system:
\begin{equation}
  \label{eq:scalar_mest_hatphatphi}
  \begin{split}
    N^+ \hat p &= \gamma_T^{-2}\int_0^T C^\star \Lambda_H Q(t) y(t) dt - \gamma_T^{-2} B_T \hat\varphi\,, \quad \delta^+\hat p = 0\,,\\
    N\hat\varphi &= B_1 Q_1^{-1}\Lambda_1 B_1^\star \hat p\,, \quad \delta\hat\varphi = B_0 Q_0^{-1} \Lambda_0 B_0^\star\gamma\hat p\,.
  \end{split}
\end{equation}
Take a total orthonormal system $\{\psi_n\}_{n\in\mathbb N}$ in $H_0$. Then $\widehat{V\varphi}=\int_0^T(\hat u,y)_Hdt=(v,\hat\varphi)_0$ and $\hat\varphi = \int_0^T\hat U(t) y(t) dt$, $\hat U:=\gamma_T^{-2} \sum_{i=1}^\infty \psi_i\otimes Q(t)C(t)p_i(t)$ provided $p_i$ solves~\cref{eq:zp} with $V^\star\Lambda\ell=\psi_i$.
\end{proposition}
\begin{proof}
\cref{c:scalar_mest} shows that in the assumptions of~\cref{p:scalar_mest_hatphi} the minimax estimate takes the following form: $\widehat{V\varphi}=\int_0^T(\hat u,y)_Hdt$. Recall from the proof of~\cref{t:vector_mest} that $\hat u = \gamma_T^{-2} Q(t) C(t) p$ provided $p$ solves~\cref{eq:zp}. Hence
\begin{equation*}
  \begin{split}
    \int_0^T(\hat u,y)_Hdt &= \gamma_T^{-2} (p, \int_0^T C^\star \Lambda_H Q(t) y(t) dt)_0 = (p,  N^+ \hat p ) + \gamma_T^{-2} (p, B_T \hat\varphi)_0\\
& = (p,  N^+ \hat p ) + \gamma_T^{-2} (B_T p, \hat\varphi)_0 = (p,  N^+ \hat p )_0 + (v, \hat\varphi)_0 - ( N^+z, \hat\varphi)_0
  \end{split}
\end{equation*}
By~\eqref{eq:Nadjoint} we get that
\begin{equation*}
  \begin{split}
(p,  N^+ \hat p )_0&-( N^+z, \hat\varphi)_0 = (\hat p, Np)_0 -(z,N\hat\varphi)_0 + \langle \gamma\hat p, \delta p\rangle-  \langle\gamma z, \delta \hat\varphi\rangle\\
& = (\hat p, B_1 Q_1^{-1}\Lambda_1 B_1^\star z)_0 - (z, B_1 Q_1^{-1}\Lambda_1 B_1^\star\hat p)_0 + \langle \gamma\hat p, B_0 Q_0^{-1} \Lambda_0 B_0^\star\gamma z\rangle - \langle\gamma z, B_0 Q_0^{-1} \Lambda_0 B_0^\star, \gamma\hat p\rangle =0\,.
\end{split}
\end{equation*}
Hence $\int_0^T(\hat u,y)_Hdt =(v, \hat\varphi)_0$. Now, we note that $v=\sum_{i=1}^\infty(v,\psi_i)_0\psi_i$ and so $p=\sum_{i=1}^\infty(v,\psi_i)_0 p_i $. But then \begin{equation*}
\int_0^T(\hat u,y)_Hdt = \gamma_T^{-2} (p, \int_0^T C^\star \Lambda_H Q(t) y(t) dt)_0 =
\gamma_T^{-2} \sum_{i=1}^\infty(v,\psi_i)_0  \int_0^T (Q(t) C p_i, y(t))_H dt =
( v, \int_0^T U(t) y(t) dt)_0
\end{equation*} hence\footnote{Recall that if $q=\sum_{j=1}^\infty \psi_i\otimes\phi_i$ provided $\psi_i$ is a total orthonormal system in $H_0$, $\phi_i$ is a total orthonormal system in $H_1$, then $\langle x, q y\rangle = \sum_{j=1}^\infty (\psi_i, x)_0 (\phi_i,y)_1$.} $\hat\varphi = \int_0^T\hat U(t) y(t) dt$.
\end{proof}
\begin{proposition}\label{p:convergence_boundedG}
Assume that $G\subset F_0\times F_1$ is bounded. If $\lim_{T\to\infty} \gamma_T^{-2}\inf_{\psi:\|\psi\|_H=1} (B_T\psi,\psi)_0=+\infty$ then $\lim_{T\to\infty} \inf_{U,c}\sup_{G,G_1} E((v,\varphi)_0-\widehat{V \varphi})^2=0$.
\end{proposition}
\begin{proof}
  Indeed, given any bounded $G\subset F_0\times F_1$ we can find $Q_{0,1}$ such that $G\subset G_0$ where $G_0$ is defined by~\cref{eq:ellipsoidG0}. Thus \(
\sup_{G,G_1} E((v,\varphi)_0-(v,\hat\varphi)_0)^2\le \sup_{G_0,G_1} E((v,\varphi)_0-(v,\hat\varphi)_0)^2
\) and so, by~\cref{c:scalar_mest}: \begin{equation*}
\inf_{U,c}\sup_{G,G_1} E((v,\varphi)_0-\widehat{V\varphi})^2\le \inf_{U}\sup_{G_0,G_1} E((v,\varphi)_0-\widehat{V\varphi} )^2 = (v,p)_0\,.
\end{equation*}
We have by~\eqref{eq:zp} that
\begin{equation*}
  \begin{split}
(v,p)_0 &= (N^+z,p)_0+\gamma_T^{-2} (B_Tp,p)_0 = (z,Np)_0 + \langle \gamma z, B_0 Q_0^{-1}\Lambda_0 B_0^\star \gamma z\rangle + \gamma_T^{-2} (B_Tp,p)_0\\
& = (z, B_1 Q_1^{-1}\Lambda_1 B_1^\star z)_0 + \langle \gamma z, B_0 Q_0^{-1}\Lambda_0 B_0^\star \gamma z\rangle + \gamma_T^{-2} (B_Tp,p)_0\ge \gamma_T^{-2} (B_Tp,p)_0\\
&\ge  \gamma_T^{-2}\inf_{\psi:\|\psi\|_H=1} (B_T\psi,\psi)_0\|p\|_0^2\,.
\end{split}
\end{equation*}
Hence $\|p\|_0\le \frac{\gamma_T^2}{\inf_{\psi:\|\psi\|_H=1} (B_T\psi,\psi)_0} \|v\|_0 $. Finally, we get \begin{equation*}
(v,p)_0\le \|v\|_0\|p\|_0\le \frac{\gamma_T^2}{\inf_{\psi:\|\psi\|_H=1} (B_T\psi,\psi)_0} \|v\|_0^2 \to0\,,\quad T\to\infty\,.
\end{equation*}
\end{proof}
\begin{proposition}\label{p:error_upper_bound_unboundedG}
Assume that $G\in F_0\times F_1$ and $\mathscr D_1:=\{U:(z,B_1f_1)_0=0\,,\quad \langle B_0f_0,\gamma z\rangle=0\}\ne\varnothing\}$. Then \begin{equation*}
\inf_{U,c}\sup_{G,G_1} E\| V\varphi - \widehat{V\varphi}\|^2_{H_V} \le \gamma_T^2\sup_{\|\ell\|_{H_V}}(\ell,Vp)_{H_V}
\end{equation*} where $p$ solves~\eqref{eq:hatp}. If, in addition, $$
\lim_{T\to\infty} \gamma_T^{-2}\inf_{\psi:\|\psi\|_H=1} (B_T\psi,\psi)_0=+\infty
$$ then $\lim_{T\to\infty} \inf_{U,c}\sup_{G,G_1} E\| V\varphi - \widehat{V\varphi}\|^2_{H_V}=0$.
\end{proposition}
\begin{proof}
Since $G\subset F_0\times F_1$ it follows that \begin{equation*}
\sigma_T^2=\inf_{U,c}\sup_{G,G_1} E\| V\varphi - \widehat{V\varphi}\|^2_{H_V}\le \sigma_1^2:=\inf_{U,c}\sup_{F_0\times F_1,G_1} E\| V\varphi - \widehat{V\varphi}\|^2_{H_V}
\end{equation*} Now, according to~\cref{eq:minsigmaG1} we get that: $\sigma_1^2 = \gamma^2_T\sup_{\ell:\|\ell\|_{V}=1}(\ell,Vp)_{H_V}$. The last statement of the proof follows from~\cref{c:unboundedG_asymptotic_convergence}.
\end{proof}
\section{Applications: pseudoinverse of a compact operator}
\label{sec:numer-exper}

Assume that $\Omega=(0,1)^2$, and let $H^k(\Omega)$ denote the Sobolev space of $L^2(\Omega)$-functions $f$ such that all weak derivatives of $f$ up to order $k$ belong to $L^2(\Omega)$, i.e. $\partial_{x1}^{\alpha_1}\partial_{x_2}^{\alpha_2}f\in L^2(\Omega)$ for any natural $\alpha_{1,2}$ such that $\alpha_1+\alpha_2\le k$. Define $H_+:=H^1(\Omega)$ and $H_0:=L^2(\Omega)$ and set \begin{equation*}
a(\phi,\psi)=\beta\int_\Omega (\nabla\phi,\nabla\psi)_{R^2} dx_1dx_2 +\mu (\phi,\psi)_{L^2(\Omega)}\,,\quad \beta\ge0\,, \mu>0\,.
\end{equation*}
In this case \begin{equation*}
N\varphi = -\beta\Delta\varphi + \mu\varphi\,, \quad H_+(N)=\{\varphi\in H^1(\Omega): \Delta\varphi\in L^2(\Omega)\} = H^2(\Omega)\,,
\end{equation*}
and the Green's formula~\eqref{eq:green1} \begin{equation*}
a(\phi,\psi) = (N\phi,\psi)_0+ \langle\delta\phi,\gamma\psi\rangle
\end{equation*} takes the following familiar form: \begin{equation*}
\begin{split}
  \int_\Omega (\nabla\phi,\nabla\psi)_{R^2} dx_1dx_2 = \int_\Omega(-\Delta)\varphi\psi dx_1dx_2+ \int_{\partial\Omega} (\nabla\phi,\vec n(x_1,x_2))_{R^2}\, \psi \,d \partial\Omega\,,
\end{split}
\end{equation*}
 where $\vec n(x_1,x_2)$ is an outward pointing normal vector to $\partial\Omega$. Note that in the considered case $\gamma$ is a standard trace operator mapping $H^1(\Omega)$ into $H_\partial=H^{\frac 12}(\Omega)\subset L^2(\partial\Omega)$, the space of traces of all $H^1(\Omega)$-functions, and $\delta$ is a bounded linear operator from $H_+(N) = H^2(\Omega)$ to $H_\partial^\star= H^{-\frac 12}(\Omega)$, the space of all linear continuous functionals over $\gamma(H^1(\Omega))$ defined as follows: \begin{equation*}
\langle \delta\phi,\gamma\psi\rangle = \int_{\partial\Omega} (\nabla\phi,\vec n(x_1,x_2))_{R^2}\, \psi \,d \partial\Omega\,, \phi\in H^1(\Omega)\,, \psi\in H^2(\Omega)\,.
\end{equation*}
Since the form $a$ is symmetric, we have that $N^+=N$ and $\delta^+=\delta$. Finally, the Neumann problem~\cref{eq:state} reads as follows: find $\varphi\in H^2(\Omega)$ such that
\begin{equation}
  \label{eq:state_laplas}
   -\beta\Delta\varphi + \mu\varphi = B_1 f_1\, \quad (\nabla\varphi,\vec n(x_1,x_2))_{R^2} = B_0 f_0\text{ on }\partial\Omega\,,
\end{equation}
This problem has the unique solution $\varphi\in H^2(\Omega)$ for any $B_1f_1\in L^2(\Omega)$ and $B_0f_0\in H^{-\frac 12}(\Omega)$ which coincides with the solution of the variational equation~\eqref{eq:neumann_var} (see~\cite[p.168]{Aubin1972}). In what follows we assume that $B_0=0$ and $B_1=I$.

We further assume that $H=H_0=L^2(\Omega)$, $Q_0=I$, $Q_1=1$ so that $G_0=\{[f_0,f_1]:\|f_0\|^2_{0}+\|f_1\|^2_0 \le 1\} $, and $Q(t)=q(t)I$, $C(t)=c(t)K$, provided $q(t)>\alpha^2>0$, and $t\mapsto c(t)$ is a given continuous function, $K\in \mathscr{L}(H_0)$. Assume that $\xi\in L^2(\Omega)$ is such that $\|\xi\|_{L^2(\Omega)}=1$. As noted in the proof of~\cref{t:unbG}, the $\sup$ over $G_1$ is attained on a ``deterministic process'' so we can take $\eta(t)=f(t)\xi(x_1,x_2)$, and take any $f:R\to R$ such that $\int_0^T q(t)f^2(t)dt\le \gamma^2_T$. Then $(Q(t)\eta(t),\eta(t))_{H}=q(t) f^2(t) \|\xi\|^2_{L^2(\Omega)} \le \gamma^2_T$. In other words, we let $G_2$ be the tensor product of the unit ball of $H_0$ with the ellipsoid $\{f: \int_0^T q f^2 dt\le 1\}$. Now, let us assume that $\varphi_{true}$ solves~\cref{eq:state_laplas} for some $f_1\in G_0=\{[f_0,f_1]:\|f_0\|^2_{0}+\|f_1\|^2_0 \le 1\}$ and we observe
\begin{equation*}
y(x_1,x_2,t) = c(t)(K \varphi_{true})(x_1,x_2) + f(t)\xi(x_1,x_2)
\end{equation*}
Let $H_V=R^1$, $V\varphi = (v,\varphi)_0$. According to~\cref{t:vector_mest,c:scalar_mest,p:scalar_mest_hatphi} we have that the minimax estimate of $V\varphi = (v,\varphi)_0$ is given by $\widehat{V\varphi}=(v,\hat\varphi)_0$ provided
\begin{equation}
  \label{eq:scalar_mest_hatphatphi_laplas}
  \begin{split}
    &N \hat p + \gamma_T^{-2}  \omega(T) K^\star K\hat\varphi= \gamma_T^{-2} \omega(T) K^\star K \varphi_{true} + \gamma_T^{-2}\int_0^T c(t) q(t) f(t) dt K^\star\xi\,, \quad \delta^+\hat p = 0\,,\\
    &N\hat\varphi = \hat p\,, \quad \delta\hat\varphi = 0\,, \omega(T):=\int_0^T c^2(t) q(t) dt\,.
  \end{split}
\end{equation}
Now, substituting $\hat p = N\hat\varphi$ into the first equation of~\cref{eq:scalar_mest_hatphatphi_laplas}, and multiplying the result by $\gamma_T^{2}\omega^{-1}(T)$ we can write: \begin{equation*}
( \frac{\gamma^2_T}{\omega(T)} N^2 + K^\star K) \hat\varphi= K^\star K \varphi_{true} + \frac{\int_0^T c(t) q(t) f(t) dt}{\omega(T)} K^\star\xi
\end{equation*}
By~\cref{p:convergence_boundedG}, we have that \begin{equation*}
\lim_{T\to\infty} \sup_{G_0,G_1} E((v,\varphi)_0- (v, \hat\varphi)_0)^2=0\,, \text{ if } \inf_{\psi:\|\psi\|_0=1} (K\psi,K\psi)_0\lim_{T\to\infty} \gamma_T^{-2}\omega(T) =+\infty
\end{equation*}
Hence, if the spectrum of the self-adjoint non-negative operator $K^\star K$ does not contain $0$, then\footnote{Note that the numerical range of $K^\star K$, defined by $\{(Kx,Kx)_0\,, \|x\|_0=1\}$ contains the spectrum of $K^\star K$, and so $\inf_{\psi:\|\psi\|_0=1} (K\psi,K\psi)_0=0$ if $0$ is in the spectrum of $K^\star K$} the expectation of $\hat\varphi$ converges weakly in $H_0$ to $\varphi$.

Note that, for compact operators $K$ of infinite rank, and more specifically for positive self-adjoint Hilbert-Schmidt integral operators, $0$ is always in the spectrum of $K^*K$, hence~\cref{p:convergence_boundedG} does not apply. However, if $\lim_{T\to\infty}\frac{\int_0^T c(t) q(t) f(t) dt}{\omega(T)}=0$ it follows that \begin{equation*}
\hat\varphi(T) = \left( \frac{\gamma^2_T}{\omega(T)} N^2 + K^\star K\right)^{-1}K^\star K \varphi_{true}+\frac{\int_0^T c(t) q(t) f(t) dt}{\omega(T)} \left( \frac{\gamma^2_T}{\omega(T)} N^2 + K^\star K\right)^{-1}K^\star\xi\,.
\end{equation*} Assume that $\beta=0$ and $\mu=1$ so that $N=I$ and define
\begin{equation}
  \label{eq:K_pseudoinverse}
q_T:=\left( \frac{\gamma^2_T}{\omega(T)} I + K^\star K\right)^{-1}K^\star K \varphi_{true} = \sum_{n=1}^\infty \frac{\lambda_n}{\frac{\gamma^2_T}{\omega(T)}+\lambda_n} (\varphi_{true},\varphi_n)_0\varphi_n
\end{equation}
where $\{\varphi_n\}$ is the total orthonormal system of eigenvectors of the compact self-adjoint operator $K^\star K$, and $\lambda_n\ge0$ are the eigenvalues of $K^\star K$: $K^\star K \varphi_n=\lambda_n \varphi_n$.  Clearly, $\frac{\lambda_n}{\frac{\gamma^2_T}{\omega(T)}}\ne 0$ if $\lambda_n> 0$ and $\frac{\lambda_n}{\frac{\gamma^2_T}{\omega(T)}}= 0$ otherwise. Hence, for $T\to\infty$ we get that $q_T\to \varphi_{true}$ provided $\min_n \lambda_n >0$ and $q_T\to \varphi^\perp_{true}$ otherwise, where $\varphi^\perp_{true}$ is the orthogonal projection of $\varphi_{true}$ onto the orthogonal completion of the null-space of $K^\star K$. Now, by using singular value decomposition of $K$, i.e. $$
Kg = \sum_n \lambda_n^\frac 12 (v_n,g)_0 u_n,\qquad K^\star g = \sum_n \lambda_n^\frac 12 (u_n,g)_0 v_n,
$$ for some total orthonormal systems $\{u_n\}$ and $\{v_n\}$ we get that $$
( \frac{\gamma^2_T}{\omega(T)} I + K^\star K)q = \sum_n \left(\frac{\gamma^2_T}{\omega(T)}+\lambda_n\right)(v_n,q)_0 v_n,\qquad K^\star \xi = \sum_n \lambda_n^\frac 12 (u_n,\xi)_0 v_n
$$ hence $$
g_T = \sum_n \frac{\lambda_n^\frac 12 }{\frac{\gamma^2_T}{\omega(T)}+\lambda_n} (u_n,\xi)_0 v_n = \left( \frac{\gamma^2_T}{\omega(T)} I + K^\star K\right)^{-1}K^\star\xi
$$ and clearly $g_T$ converges to a vector in $H$ for $T\to \infty$. Thus, for the worst-case realisation of $\xi$, we have that
\begin{equation*}
\frac{\int_0^T c(t) q(t) f(t) dt}{\omega(T)} \left( \frac{\gamma^2_T}{\omega(T)} I + K^\star K\right)^{-1}K^\star\xi\to 0\,.
\end{equation*}
More specifically, if $K$ is of Volterra type, so that the null-space of $K$ is trivial, then the algorithm \begin{equation*}
T\mapsto q_T + \frac{\int_0^T c(t) q(t) f(t) dt}{\omega(T)} \left( \frac{\gamma^2_T}{\omega(T)} I + K^\star K\right)^{-1}K^\star\xi
\end{equation*} ``differentiates'' the noisy signal $y$ and converges to $\varphi_{true}$, the ``derivative'' of $K \varphi_{true}$.

\subsubsection{Numerical experiment: positive Hilbert-Schmidt operator}
\label{sec:poisson-equations}

For the numerical experiment we took $K\varphi = \int_\Omega e^{-\frac{(x-x')^2+(y-y')^2}{2*10^2}}\varphi(x',y')dx'dy'$, $q(t)=1$, $\mu:=10^{-4}$ and defined $c$ and $f$ as follows:
\begin{itemize}
\item $c(t)=1$ if $t\in [n-n^{-\frac12}, n]$ for any natural $n$ and $c(t)=0$ otherwise
\item $f(t) = n^{-2} \frac{t-n+n^{-\frac 12}}{n^{-\frac 12} - n^{-2}}$ if $t\in [n-n^{-\frac 12}, n-n^{-2}]$, and $f(t) = (1-n^{-2}) \frac{t-n+n^{-2}}{n^{-2}} + n^{-2}$ if $t\in [n-n^{-2},n]$, and $f=0$ otherwise
\end{itemize}
Clearly, $\int_0^{+\infty}c^2(t) dt = \sum_{n=1}^{+\infty}\frac1{\sqrt{n}}=+\infty$, and $f(n)=1$, and \begin{equation*}
\int_0^{+\infty} f(t) dt = \sum_{n=1}^{+\infty} \frac{1}{2 n^2}+\frac{n^{3/2}}{2 n^4} = \frac{1}{12} \left(6 \zeta \left(\frac{5}{2}\right)+\pi ^2\right)\approx 1.49\,,
\end{equation*} where $\zeta$ denotes the Riemann zeta function. Moreover, \begin{equation*}
\int_0^{+\infty} f^2(t) dt = \sum_{n=1}^{+\infty} \frac{1}{3 n^4}+\frac{1}{3 n^2}+\frac{n^{3/2}}{3 n^6} = \frac{1}{270} \left(90 \zeta \left(\frac{9}{2}\right)+\pi ^4+15 \pi ^2\right) \approx 1.26\,.
\end{equation*}
It then follows that $\int_0^T c(t) q(t) f(t) dt\le 1.49$ and we can set $\gamma_T^2=1.26$. We take \begin{equation*}
f_1:=100\cos(2\pi x)\cos(3\pi y) + 0.1 \exp\{ -\frac{(x - 0.5)^2+(y-0.5)^2}{0.05^2}\}
\end{equation*} to generate the ``true'' solution $\varphi_{true}$ of~\cref{eq:state_laplas} (see~\cref{fig:truth}). This solution has been approximated by Spectral Element method with $60\times 60$ elements and 2nd order Lagrange interpolation polynomials over Gauss-Legendre-Lobatto quadrature points~\cite{SEM_CFD}. It was then used to generate the observations $y$~\cref{fig:obs}. Finally, the minimax estimate has been computed for $T_1$ such that $\omega(T) = \int_0^{T}c^2(t) dt = 23\times 10^6$ and so
\begin{equation*}
\hat\varphi(T_1) = \left( \frac{1.26}{23}\times 10^{-6} N^2 + K^\star K\right)^{-1}K^\star K \varphi_{true}+\frac{1.49}{23}\times 10^{-6} \left( \frac{1.26}{23}\times 10^{-6} N^2 + K^\star K\right)^{-1}K^\star\xi\,.
\end{equation*}
The estimate is presented in~\cref{fig:mest}.
\begin{figure}
        \centering
        \begin{subfigure}[b]{0.5\textwidth}
                \includegraphics[width=\textwidth]{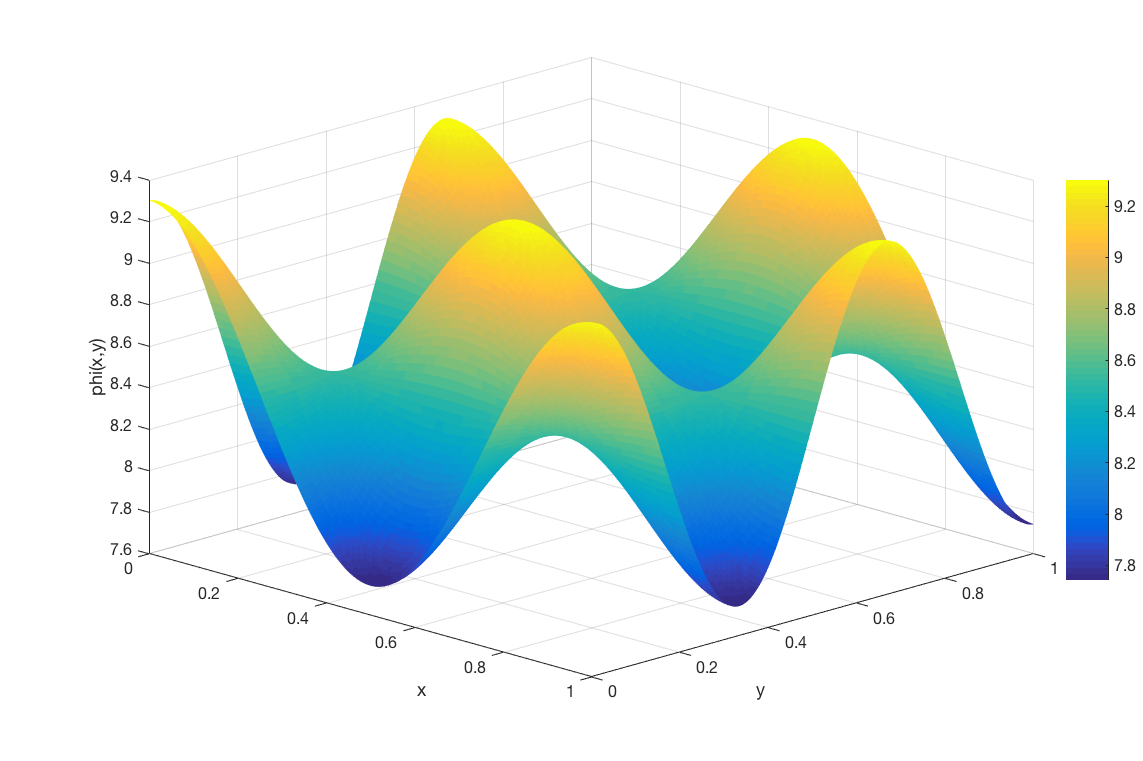}
                \caption{Truth}
                \label{fig:truth}
        \end{subfigure}%
        ~ 
\begin{subfigure}[b]{0.5\textwidth}
                \includegraphics[width=\textwidth]{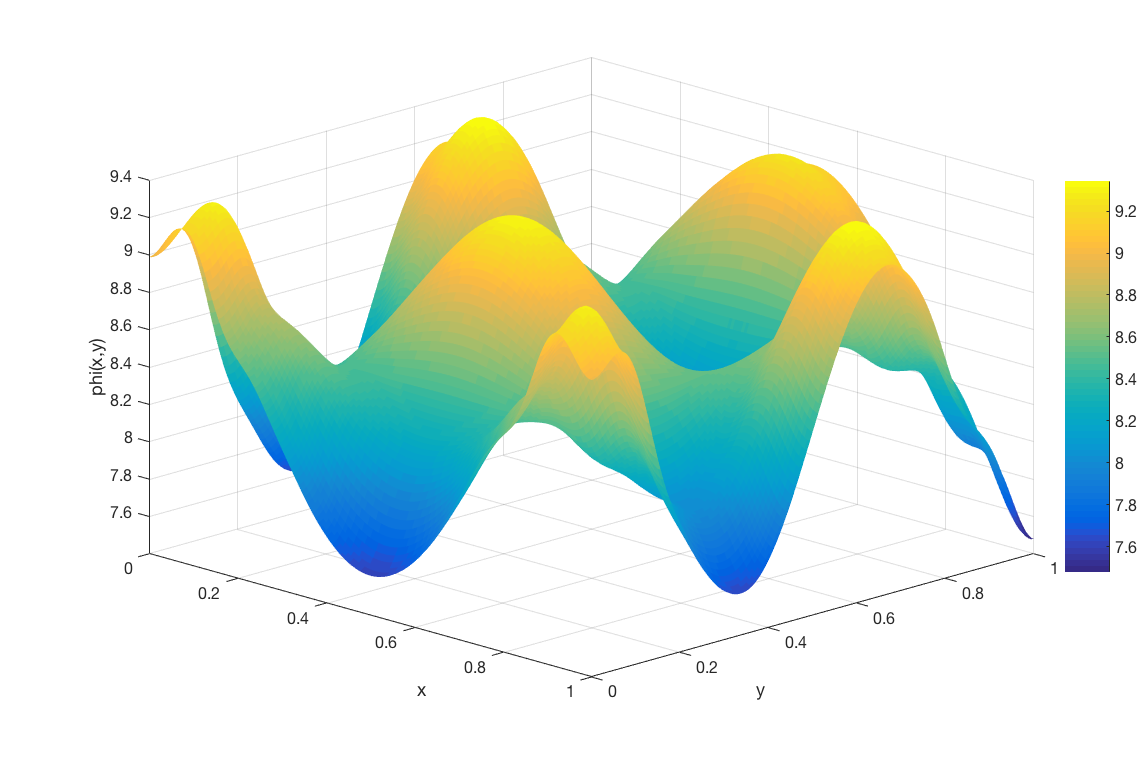}
                \caption{Minimax estimate (relative $L^2$-error $\approx$ 1\%) }
                \label{fig:mest}
        \end{subfigure}%

        \begin{subfigure}[b]{0.5\textwidth}
                \includegraphics[width=\textwidth]{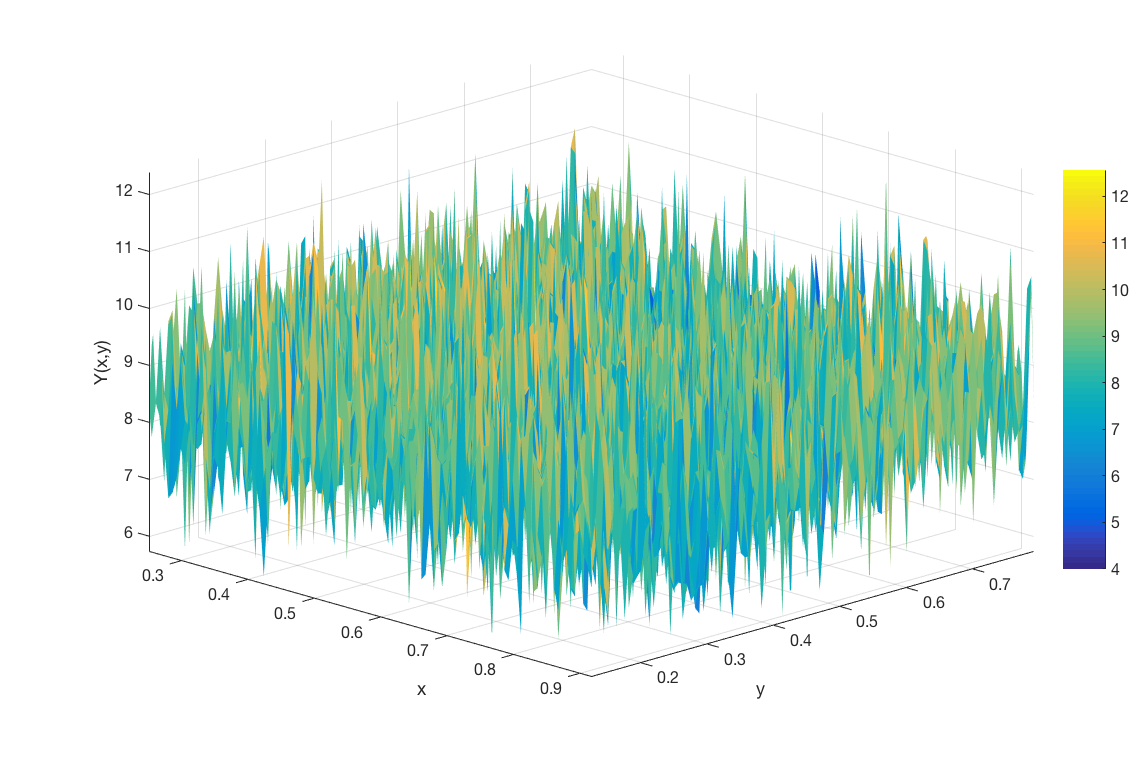}
                \caption{Observations}
                \label{fig:obs}
        \end{subfigure}%
        ~ 
        \begin{subfigure}[b]{0.45\textwidth}
                \includegraphics[width=\textwidth]{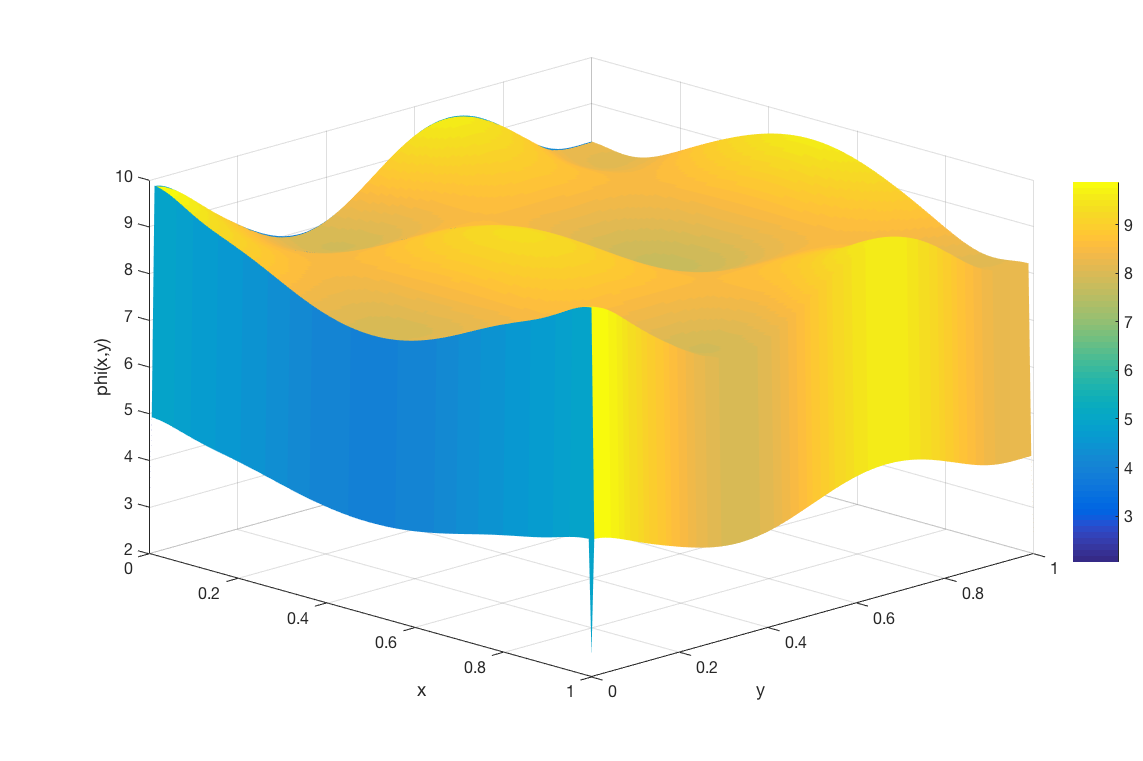}
                \caption{Pseudo-inverse: $N=I$, no noise in observations, (relative $L^2$-error circa 8\%).}
                \label{fig:pinv}
        \end{subfigure}%
        ~ 
	\caption{Truth, observations, estimate and pseudoinverse}
\end{figure}

\section{Conclusion}
\label{sec:conclusion}

The paper provides representations for the minimax estimates of solutions of abstract Neumann problems with uncertain inputs/boundary conditions given incomplete and noisy measurements of the solution. The exact expressions for the minimax estimate are derived for the case of estimating a component of the state vector, i.e. a linear continuous functional of the state vector (see~\cref{p:scalar_mest_hatphi}). For a more general case of estimating the entire solution, or a linear operator of the solution,  a non-trivial lower bound for the minimax error is found (see~\cref{t:vector_mest}). Finally, it is demonstrated that the minimax estimate is asymptotically exact under mild assumptions on observations and bounding sets as suggested in~\cref{p:convergence_boundedG}. An interesting relationship between the proposed estimates and robust pseudoinversion of compact operators is revealed in~\cref{sec:numer-exper}: it turns out that, for $T\to\infty$ the minimax estimate converges to the solution of the least-squares problem $\|Kx-K\varphi_{true}\|_0\to\min_x$ with the minimal norm. A promising research direction would be to use the proposed results to design stochastic differentiators, i.e. algorithms which can compute time/partial derivatives of the noisy signals.

\end{document}